\documentclass[reqno,11pt]{amsart}

\usepackage[margin=1in]{geometry}
\usepackage{xspace}
\usepackage[utf8]{inputenc}
\usepackage{graphicx}
\usepackage{amssymb}
\usepackage{mathrsfs}
\usepackage[active]{srcltx}
\usepackage{url}
\usepackage{hyperref}
\usepackage{amsmath,amstext,amsxtra,amsgen,amsbsy,amsopn,amscd,amsthm,amsfonts}
\usepackage{latexsym}
\usepackage{dsfont}
\usepackage{enumitem}
\usepackage{tikz-cd}
\usepackage{tikz}

\newtheorem{theorem}{Theorem}

\newtheorem{proposition}{Proposition}
\newtheorem{lemma}[proposition]{Lemma}

\theoremstyle{definition}

\theoremstyle{remark}
\newtheorem{remark}[proposition]{Remark}

\numberwithin{equation}{section}
\numberwithin{proposition}{section}

\begin{document}

\title[Bochner-Riesz operator of negative order ]{Negative Order Bochner-Riesz Operators for the Critical Magnetic Schr\"odinger Operator in $\mathbb{R}^2$}

\author{Huanqing Guo}
\address{Huanqing Guo
\newline \indent The Graduate School of China Academy of Engineering Physics, Beijing, 100088,\  China}
\email{guohuanqing23@gscaep.ac.cn}

\author{Junyong Zhang}
\address{Junyong Zhang
\newline \indent Department of Mathematics, Beijing Institute of Technology, Beijing 100081}
\email{zhang\_junyong@bit.edu.cn}

\author{Jiiqang Zheng}
\address{Jiqiang Zheng
\newline \indent Institute of Applied Physics and Computational Mathematics,
Beijing, 100088, China.
\newline\indent
National Key Laboratory of Computational Physics, Beijing 100088, China}
\email{zheng\_jiqiang@iapcm.ac.cn, zhengjiqiang@gmail.com}

\begin{abstract}
This paper studies the sharp $L^p$–$L^q$ boundedness of the Bochner–Riesz operator $S^{\delta}_{\lambda}(\mathcal{L}_{\mathbf{A}})$ associated with a scaling-critical magnetic Schr\"odinger operator $\mathcal{L}_{\mathbf{A}}$ on $\mathbb{R}^2$, where $\delta \in (-3/2, 0)$. We determine the conditions on the exponents $p$ and $q$ under which the operator is bounded from $L^p(\mathbb{R}^2)$ to $L^q(\mathbb{R}^2)$. Our main result characterizes the boundedness region as a pentagonal subset $\Delta(\delta)$ of the $(1/p, 1/q)$-plane, extending previous uniform resolvent result in Fanelli, Zhang and  Zheng[Int. Math. Res. Not., 20(2023), 17656-17703].
\end{abstract}

\maketitle

\begin{center}
 \begin{minipage}{140mm}
 \end{minipage}
 \end{center}

\section{Introduction}
Consider a scaling-invariant magnetic Schr\"odinger operator on \(\mathbb{R}^2 \setminus \{0\}\) defined by
\begin{equation}\label{LA}
\mathcal{L}_{\mathbf{A}} = -\left( \nabla + i \frac{\mathbf{A}(\hat{x})}{|x|} \right)^2, \quad x \in \mathbb{R}^2 \setminus \{0\},
\end{equation}
where \(\hat{x} = x/|x| \in \mathbb{S}^1\) and the vector field \(\mathbf{A} \in W^{1,\infty}(\mathbb{S}^1; \mathbb{R}^2)\) satisfies the transversality condition: $	\mathbf{A}(\hat{x}) \cdot \hat{x} = 0$, for all $x \in \mathbb{R}^2$. We study the Bochner–Riesz operator of order \(\delta\) associated with \(\mathcal{L}_{\mathbf{A}}\), which is defined by
\begin{eqnarray} \label{oper:BR}
S_{\lambda}^\delta(\mathcal{L}_\mathbf{A})
:=\frac{1}{\Gamma(1+\delta)}\int_0^{\lambda} \Big(1-\frac{\rho^2}{\lambda^2}\Big)^\delta  dE_{\sqrt{\mathcal{L}_{\mathbf{A}}}} (\rho), 
\end{eqnarray}
where \(\lambda > 0\) and \(E_{\sqrt{\mathcal{L}_{\mathbf{A}}}}(\rho)\) is the spectral resolution of \(\sqrt{\mathcal{L}_{\mathbf{A}}}\). This work is part of a broader research program studying such operators, see, e.g., \cite{FZZ, FZZ1, GWZZ, GYZZ}.

The study of Bochner–Riesz operators originated in the context of the Laplacian \(\Delta = -\sum\limits_{i=1}^n \partial_{x_i}^2\) on \(\mathbb{R}^n\), \(n \geq 2\). A central conjecture asserts that for \(\delta > 0\), the Bochner–Riesz mean \(S^{\delta}_{\lambda}(-\Delta)f\) converges in \(L^p(\mathbb{R}^n)\) if and only if
\[
\delta > \delta_c(p,n) = \max\left\{0, \, n\left| \tfrac{1}{2} - \tfrac{1}{p} \right| - \tfrac{1}{2} \right\}, \quad p \in [1, \infty].
\]
This conjecture was resolved in two dimensions by Carleson and Sjölin \cite{CS} via a fundamental result on oscillatory integral operators. Alternative approaches were later given by Hörmander \cite{Hor}, Fefferman \cite{Fef2}, and Córdoba \cite{Cor}. In higher dimensions, only partial results are known; see \cite{Bour, BG, GOWWZ, Lee} and references therein.

Bochner–Riesz means for Schr\"odinger operators with potentials have also been widely studied. For example, Lee and Ryu \cite{LR} treated the case of the Hermite operator, and Jeong, Lee, and Ryu \cite{JLR} studied the twisted Laplacian in \(\mathbb{R}^2\). More recently, Miao, Yan, and Zhang \cite{MYZ} considered Bochner–Riesz means for Schr\"odinger operators with this scaling-critical magnetic potential. They showed that for \(\delta > 0\),
\[
\| S_{\lambda}^{\delta}(\mathcal{L}_{\mathbf{A}}) f \|_{L^p(\mathbb{R}^2)} \leq C \| f \|_{L^p(\mathbb{R}^2)}
\]
holds with \(C > 0\) independent of \(f\) if and only if \(\delta > \delta_c(p,2)\).

On the other hand, the $L^p$–$L^q$ boundedness of Bochner–Riesz operator of negative order has been investigated by several authors (e.g., \cite{Sogge, CS0, B, CKLS}). In the two-dimensional case, Jong-Guk Bak \cite{B} proved that for \(0 < \delta < \tfrac32\),
\[\| S_1^{-\delta}(-\Delta) f \|_{L^q(\mathbb{R}^2)} \leq C \| f \|_{L^p(\mathbb{R}^2)}\]
holds with \(C = C(p,q,\delta)\) independent of \(f\) if and only if \((1/p, 1/q) \in \Delta(\delta)\), where
\[
\Delta(\delta) = \left\{ \left( \frac{1}{p}, \frac{1}{q} \right) \in [0,1]^2 : \frac{1}{p} - \frac{1}{q} \geq \frac{2\delta}{3},\ \frac{1}{p} > \frac14 + \frac{\delta}{2},\ \frac{1}{q} < \frac34 - \frac{\delta}{2} \right\}.
\]

Geometrically, for fixed \(0 < \delta < \tfrac32\), the region \(\Delta(\delta)\) is the closed solid pentagon \(ABB'A'D\) in the unit square, excluding the closed segments \(AB\) and \(A'B'\), where the points are defined as follows:
\[
A = \left( \tfrac14 + \tfrac{\delta}{2}, 0 \right), \;
A' = \left( 1, \tfrac34 - \tfrac{\delta}{2} \right), \;
B = \left( \tfrac14 + \tfrac{\delta}{2}, \tfrac14 - \tfrac{\delta}{6} \right), \;
B' = \left( \tfrac34 + \tfrac{\delta}{6}, \tfrac34 - \tfrac{\delta}{2} \right),\; D=(1,0).
\]
Note that \(B\) lies on the line \(\tfrac1p +\tfrac3q = 1\); see Figure \ref{fig:AB}.

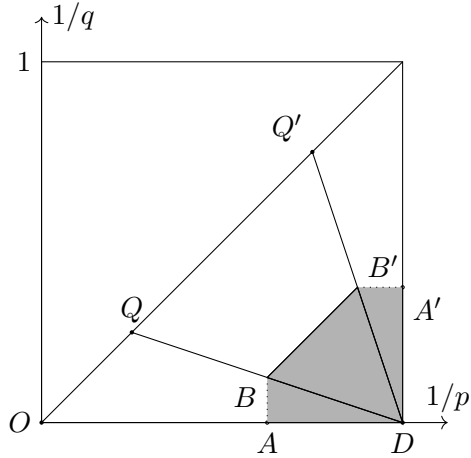
\begin{figure}[htbp] \label{fig:AB}
\begin{center}
\begin{tikzpicture}[scale=0.6]
\draw (0,0) rectangle (8,8);
\draw[->]  (0,0) -- (0,9);
\draw[->]  (0,0) -- (9,0);
				
\filldraw (0,0) circle (1pt) node[left] {$O$};  
\filldraw (8,0) circle (1pt) node[below ] {$D$};
\filldraw (6,6) circle (1pt) node[above left] {$Q'$};
\filldraw (2,2) circle (1pt) node[above ] {$Q$}; 
			
\draw (5,0) circle (1pt) node[below ] {$A$};    
\draw (8,3) circle (1pt) node[below right] {$A'$};
\draw (0,8) node[left] {$1$};
\draw (9,0) node[above] {$1/p$};
\draw (0,9) node[right] {$1/q$};
\draw (7,3) node[above right] {$B'$};
\draw (5,1) node[below left] {$B$};
			
\filldraw[fill=gray!60](5,0)--(8,0)--(5,1); 
\filldraw[fill=gray!60](7,3)--(8,0)--(8,3); 
\filldraw[fill=gray!60](8,0)--(7,3)--(5,1)--(8,0); 
			
\draw(0,0) -- (8,8);
\draw(8,0) -- (2,2);
\draw(8,0) -- (6,6);
\draw[dotted] (5,0) -- (5,1);
\draw[dotted] (8,3) -- (7,3);
\draw (7,3) -- (5,1);
			
\end{tikzpicture}
\end{center}
\caption[ ]{the region $\Delta(\delta)$}

\end{figure}

In this paper, we study the analogous \(L^p\)–\(L^q\) boundedness problem for the Bochner–Riesz operator of negative order associated with the magnetic Schr\"odinger  operator \(\mathcal{L}_{\mathbf{A}}\). Although the spectral definition \eqref{oper:BR} is initially valid for \(\Re \delta > -1\), it extends to all \(\delta \in \mathbb{C}\) via analytic continuation. To see this, let's define $A(\theta):[0,2\pi)\to \mathbb{R}$ such that
\begin{equation}\label{equ:alpha}
A(\theta)={\bf A}(\cos\theta,\sin\theta)\cdot (-\sin\theta,\cos\theta),
\end{equation}
and we define the constant $\alpha$ to be
$$\alpha=\Phi_\mathbf{A}=\frac1{2\pi}\int_0^{2\pi} A(\theta) \, d \theta,$$
which is regarded as the total magnetic flux of the magnetic field.\vspace{0.2cm}

Let $\lambda>0$, we define the resolvent of the self-adjoint operator $\mathcal{L}_\mathbf{A}$ by
\begin{equation}\label{def:res}
	\big(\mathcal{L}_{\mathbf{A}}-(\lambda^2\pm i0)\big)^{-1}
	=\lim_{\epsilon\searrow0}\big(\mathcal{L}_{\mathbf{A}}-(\lambda^2\pm i\epsilon)\big)^{-1},
\end{equation}
where we use the same notation $\mathcal{L}_\mathbf{A}$ to denote its Friedrichs self-adjoint extension of the Hamiltonian \eqref{LA}. In \cite{GYZZ}, the authors established the resolvent kernel as follows: 
Let $x=r_1(\cos\theta_1,\sin\theta_1)$ and $y=r_2(\cos\theta_2,\sin\theta_2)$, then we have the expression of resolvent kernel
\begin{align}\label{equ:res-ker-out}
\big(\mathcal{L}_{\mathbf{A}}-(\lambda^2\pm i0)\big)^{-1}(x,y)=&\frac1{\pi}\int_{\mathbb{R}^2}\frac{e^{-i(x-y)\cdot\xi}}{|\xi|^2-(\lambda^2\pm i0)}\;d\xi\, A_\alpha(\theta_1,\theta_2)\\\nonumber
&+\frac1{\pi}\int_0^\infty \int_{\mathbb{R}^2}\frac{e^{-i{\bf n}_s\cdot\xi}}{|\xi|^2-(\lambda^2\pm i0)}\;d\xi \, B_\alpha(s,\theta_1,\theta_2)\;ds,
\end{align}
where  ${\bf n}_s=(r_1+r_2, \sqrt{2r_1r_2(\cosh s-1)})$ and where
\begin{equation}\label{A-al}
\begin{split}
&A_{\alpha}(\theta_1,\theta_2)= \frac{e^{i\int_{\theta_1}^{\theta_2}\,A(\theta')d\theta'}}{4\pi^2}\times \Big(\mathbf{1}_{[0,\pi]}(|\theta_1-\theta_2|)
+e^{-i2\pi\alpha}\mathbf{1}_{[\pi,2\pi]}(\theta_1-\theta_2)+e^{i2\pi\alpha}\mathbf{1}_{[-2\pi,-\pi]}(\theta_1-\theta_2)\Big),
\end{split}
\end{equation}
and
\begin{equation}\label{B-al}
\begin{split}
&B_{\alpha}(s,\theta_1,\theta_2)= -\frac{1}{4\pi^2}e^{-i\alpha(\theta_1-\theta_2)+i\int_{\theta_2}^{\theta_{1}} A(\theta') d\theta'}  \Big(\sin(|\alpha|\pi)e^{-|\alpha|s}\\
&\qquad +\sin(\alpha\pi)\frac{(e^{-s}-\cos(\theta_1-\theta_2+\pi))\sinh(\alpha s)-i\sin(\theta_1-\theta_2+\pi)\cosh(\alpha s)}{\cosh(s)-\cos(\theta_1-\theta_2+\pi)}\Big).
\end{split}
\end{equation}

According to Stone's formula, the spectral measure is related to the resolvent
\begin{equation}\label{equ:spemes}
dE_{\sqrt{\mathcal{L}_\mathbf{A}}}(\lambda)=\frac{d}{d\lambda}E_{\sqrt{\mathcal{L}_\mathbf{A}}}(\lambda)\;d\lambda
=\frac{\lambda}{i\pi}\big(R(\lambda+i0)-R(\lambda-i0)\big)\;d\lambda
\end{equation}
where the resolvent is given by
$$R(\lambda\pm i0)=\big(\mathcal{L}_{\mathbf{A}}-(\lambda^2\pm i0)\big)^{-1}=\lim_{\epsilon\searrow0}\big(\mathcal{L}_{\mathbf{A}}-(\lambda^2\pm i\epsilon)\big)^{-1}.$$
Let $x=r_1(\cos\theta_1, \sin\theta_1)$ and $y=r_2(\cos\theta_2, \sin\theta_2)$ in $\mathbb{R}^2\setminus\{0\}$. It follows from \eqref{equ:res-ker-out}
 and \eqref{equ:spemes} that the Schwartz kernel of the spectral measure satisfies
\begin{equation}\label{ker:spect}
dE_{\sqrt{\mathcal{L}_{\mathbf{A}}}}(\lambda;x,y) =\frac{\lambda}{\pi} \Big(
\int_{\mathbb{S}^1} e^{-i\lambda (x-y)\cdot\omega} d\sigma_\omega A_{\alpha}(\theta_1,\theta_2)
+\int_0^\infty \int_{\mathbb{S}^1} e^{-i\lambda {\bf n}_s\cdot\omega} d\sigma_\omega
B_{\alpha}(s,\theta_1,\theta_2) ds\Big),
\end{equation}
where $A_{\alpha}(\theta_1,\theta_2)$ and $B_{\alpha}(s,\theta_1,\theta_2)$ are given in \eqref{A-al} and \eqref{B-al}. Then the kernel of $S^{\delta}_{\lambda}(\mathcal{L}_{\mathbf{A}})$ can be represented as 
\begin{equation}\label{ker:BR}
\begin{split}
\Big(1-\frac{\mathcal{L}_\mathbf{A}}{\lambda^2}\Big)^\delta_{+}(x,y)
:=&\frac{1}{\Gamma(1+\delta)}\int_0^\infty \Big(1-\frac{\rho^2}{\lambda^2}\Big)^\delta_{+} dE_{\sqrt{\mathcal{L}_{\mathbf{A}}}}(\rho;x,y)  \, d\rho\\
=&\lambda^2 \pi^{-\delta}(2\pi)^{1+\delta}(\lambda|x-y|)^{-1-\delta} J_{1+\delta}(\lambda|x-y|) A_{\alpha}(\theta_1,\theta_2)\\
&+\lambda^2 \pi^{-\delta}(2\pi)^{1+\delta}\int_0^{\infty}(\lambda|\mathbf{n}_s|)^{-1-\delta} J_{1+\delta}(\lambda|\mathbf{n}_s|)\, B_{\alpha}(s,\theta_1,\theta_2)ds
\end{split}
\end{equation}
for $\Re \delta>-1$. Thus we have
\begin{equation}\label{oper:BR-delta}
[S_{\lambda}^\delta(\mathcal{L}_\mathbf{A})f](x)=\int_{\mathbb{R}^2}\Big(1-\frac{\mathcal{L}_\mathbf{A}}{\lambda^2}\Big)^\delta_{+}(x,y)f(y)dy.
\end{equation}
For our present purposes, we will define $S^{\delta}_{\lambda}(\mathcal{L}_{\mathbf{A}})$ for $\delta\in (-\frac{3}{2},0)$ by \eqref{oper:BR-delta} and \eqref{ker:BR}.

In this paper, we determine all pairs $(p,q)$ such that the operator $S^{\delta}_{\lambda}(\mathcal{L}_{\mathbf{A}})$ on $\mathbb{R}^2$ $(-\frac{3}{2}<\delta<0)$ is bounded from $L^p(\mathbb{R}^2)$ to $L^q(\mathbb{R}^2)$. Our main result is the following.

\begin{theorem}\label{thm:LA0}
Let $0<\delta<\frac{3}{2}$ and let  $S^{-\delta}_\lambda(\mathcal{L}_{\mathbf{A}})$ be the Bochner-Riesz operator defined by
\eqref{oper:BR-delta} and \eqref{ker:BR}.
Then, there exists  a constant $C=C(p,q,\delta)$ independent of $f$ such that 
\begin{equation}\label{est:BR}
\|S^{-\delta}_{\lambda}(\mathcal{L}_{\mathbf{A}})f\|_{L^q(\mathbb{R}^2)}\leq C\lambda^{2\left(\frac{1}{p}-\frac{1}{q}\right)}\|f\|_{L^p(\mathbb{R}^2)},
\end{equation}
if and only if $(1/p,1/q)\in \Delta(\delta)$.
\end{theorem}

\begin{remark}
In the free context $\mathbf{A}\equiv 0$, the above result corresponds to \cite{B}. When $\delta=1$, Theorem \ref{thm:LA0} is essentially the uniform resolvent estimates. And the corresponding results in this direction were obtained in \cite{FZZ1}. Through the theory of intertwining operators established in \cite{FSWZZ}, we know that the same $L^p-L^q$ boundedness results hold for Bochner-Riesz operator associated with the following electromagnetic Schr\"{o}dinger operator on $\mathbb{R}^2$:
\[\mathcal{L}_{\mathbf{A},a}:=\left(-i\nabla+\frac{\mathbf{A}(\hat{x})}{|x|}\right)^2+\frac{a(\hat{x})}{|x|^2},\]
where $\hat{x},\mathbf{A}$ are defined as before, and $a\in W^{1,\infty}(\mathbb{S}^1)$. 
\end{remark}

\begin{remark}
The necessary part of \eqref{est:BR} follows by a transplantation theorem due to Kenig-Stanton-Tomas\cite{KST} and the necessary condition for $L^p-L^q$ boundedness on the classicial Bochner-Riesz operator for $S^{-\delta}_{\lambda}(-\Delta)$.
\end{remark}

This paper is organized as follows. In Section \ref{sec:thmmain}, we derive an explicit expression for the kernel of the Bochner–Riesz operator $S_{\lambda}^\delta(\mathcal{L}_\mathbf{A})$. As indicated in \eqref{ker:BR}, the kernel splits naturally into two components, denoted by the geometry term $G$ and  the diffractive term $D$. By exploiting this decomposition, we reduce the proof of Theorem \ref{thm:LA0} to establishing certain restricted weak-type estimates for the operators corresponding to $G$ and $D$. The analysis of the $G$-component is presented in Section \ref{sec:TG}, while the estimates for the $D$-component are treated in Section \ref{sec:TD}.

It is worth noting that, in dealing with these estimates, the techniques developed in \cite{FZZ1} and \cite{MYZ} suffice for the $D$-component. For the $G$-component, however, we provide a substantially simpler argument by establishing a ``stability lemma" for an associated oscillatory integral operator; see Proposition \ref{prop:TGj} for a precise statement.

\section{Some Standard Reductions}
\label{sec:thmmain}
In this section, we reduce the proof Theorem \ref{thm:LA0} to proving some restricted weak type estimates. By the scaling invariant of the operator $\mathcal{L}_{\mathbf{A}}$, it suffices to prove \eqref{est:BR} when $\lambda= 1$, that is,
\begin{equation}\label{est:BR1}
\|S^{-\delta}_{1}(\mathcal{L}_{\mathbf{A}})f\|_{L^q(\mathbb{R}^2)}\leq C\|f\|_{L^p(\mathbb{R}^2)},
\end{equation}
where $p,q$ are as in Theorem \ref{thm:LA0}.
Indeed,  from  the definition, one can see that
$S_{\lambda}^{-\delta}(\mathcal{L}_\mathbf{A})(x,y)=\lambda^2 S_{1}^{-\delta}(\mathcal{L}_\mathbf{A})(\lambda x, \lambda y),$
hence \eqref{est:BR1} implies \eqref{est:BR}. We briefly write $S^{-\delta}(\mathcal{L}_\mathbf{A})=S_{1}^{-\delta}(\mathcal{L}_\mathbf{A})$.

Let us recall that Bochner-Riesz operator of order $\delta$  are defined by \eqref{oper:BR-delta}. We have the following proposition.

\begin{proposition}
	\label{prop: ker-BR}
	Let $x=r_1(\cos\theta_1,\sin\theta_1)$, $y=r_2(\cos\theta_2,\sin\theta_2)$ and $0<\delta<\tfrac32$.
	Define
	\begin{equation}\label{d-j}
		d(r_1,r_2,\theta_1,\theta_2)=\sqrt{r_1^2+r_2^2-2r_1r_2\cos(\theta_1-\theta_2)}=|x-y|,
	\end{equation}
	and
	\begin{equation} \label{d-s}
		d_s(r_1,r_2,\theta_1,\theta_2)=|{\bf n}_s|=\sqrt{r_1^2+r_2^2+2  r_1r_2\, \cosh s},\quad s\in [0,+\infty).
	\end{equation}
	Then the kernel of the Bochner-Riesz operator can be written as
	\begin{equation}
		\label{ker:BR-0}
		\Big(1-\mathcal{L}_{\mathbf{A}}\Big)^{-\delta}_{+}(x,y)
		=G(\delta; r_1,\theta_1;r_2,\theta_2)+ D(\delta; r_1,\theta_1;r_2,\theta_2).
	\end{equation}
	Here
	\begin{equation}
		\label{ker:BR-G}
		\begin{split}
			G(\delta; r_1,\theta_1;r_2,\theta_2)
			&\sim \sum_{j=0}^{\infty}\Big[\frac{a^j_+e^{i d}}{ d^{\frac{3}2-\delta+j}}+\frac{a^j_-e^{-i d}}{ d^{\frac{3}2-\delta+j}}\Big] \times
			A_{\alpha}(\theta_1,\theta_2)
		\end{split}
	\end{equation}
	as $d\to\infty$, and
	\begin{equation}
		\label{ker:BR-D}
		\begin{split}
			D(\delta; r_1,\theta_1;r_2,\theta_2)
			&\sim\int_0^\infty \sum_{j=0}^{+\infty}\Big[\frac{a^j_+e^{i d_s}}{d_s^{\frac{3}2-\delta+j}}+\frac{a^j_-e^{-i d_s}}{d_s^{\frac{3}2-\delta+j}}\Big]  \, B_{\alpha}(s,\theta_1,\theta_2) ds
		\end{split}
	\end{equation}
	as $d_s\to\infty$, where the $a^j_{\pm}$ are suitable coefficients.
\end{proposition}

\begin{proof}
	Let us define
	\begin{equation}
		G(\delta; r_1,\theta_1;r_2,\theta_2)
		=\pi^{\delta}(2\pi)^{1-\delta}|x-y|^{-1+\delta} J_{1-\delta}(|x-y|) A_{\alpha}(\theta_1,\theta_2)
	\end{equation}
	and
	\begin{equation}
		D(\delta; s; r_1,\theta_1;r_2,\theta_2)
		= \pi^{\delta}(2\pi)^{1-\delta}\int_0^{\infty}|\mathbf{n}_s|^{-1+\delta} J_{1-\delta}(|\mathbf{n}_s|)
		\, B_{\alpha}(s,\theta_1,\theta_2)ds.
	\end{equation}
	Then from \eqref{ker:BR}, it suffices to prove \eqref{ker:BR-G} and \eqref{ker:BR-D}. We have the complete asymptotic expansion of Bessel function $J_{\nu}(r)$ (e.g. \cite[(15) Chapter 8, P338]{Stein})
	$$J_{\nu}(r)\sim r^{-\frac12}e^{ir}\sum_{j=0}^\infty a^j_+ r^{-j}+r^{-\frac12}e^{-ir}\sum_{j=0}^\infty a^j_- r^{-j}, \quad \text{as}\,\, r\to+\infty$$
	for suitable coefficients $a^j_{\pm}$. Therefore, these imply that  \eqref{ker:BR-G} and \eqref{ker:BR-D} hold. 
\end{proof}

First, a direct computation yields the following facts (proved in \cite{FZZ}) that
\begin{align}\label{equ:ream1}
	|x-y|\lesssim|&\mathbf{n}_s|,\\\label{equ:ream2}
	\int_0^\infty e^{-|\alpha|s}\;ds\lesssim&1,\\\label{equ:ream3}
	\int_0^\infty  \Big|\frac{(e^{-s}-\cos(\theta_1-\theta_2+\pi))\sinh(\alpha s)}{\cosh(s)-\cos(\theta_1-\theta_2+\pi)}\Big|\;ds\lesssim&1,\\\label{equ:ream4}
	\int_0^\infty  \Big|\frac{\sin(\theta_1-\theta_2+\pi)\cosh(\alpha s)}{\cosh(s)-\cos(\theta_1-\theta_2+\pi)}\Big|\;ds\lesssim&1.
\end{align}
Using \eqref{ker:BR-0} and \eqref{equ:ream1}-\eqref{equ:ream4}, we easily obtain that
$$\Big|\Big(1-\frac{\mathcal{L}_\mathbf{A}}{\lambda^2}\Big)^{-\delta}_{+}(x,y)\Big|\lesssim \langle |x-y|\rangle^{-\frac32+\delta}.$$
This together with Young's inequality implies 
$$\|S^{-\delta}(\mathcal{L}_{\mathbf{A}})f\|_{L^\infty(\mathbb{R}^2)}\leq C\|f\|_{L^p(\mathbb{R}^2)},\;\frac1p>\frac14+\frac{\delta}2.$$
And so  \eqref{est:BR1}  holds true on the line $AD$. Thus, by interpolation, it suffices to prove \eqref{est:BR1} when $(1/p, 1/q)$ lies on the open segment $BB'$, that is, $1/p-1/q=2\delta/3$ and $1/4-\delta/6<1/q< 3/4-\delta/2$, for $0<\delta<3/2$\footnote{For convenience, in this paper we will denote the open segment $BB'$ by $(BB')$.}. To prove this, it suffices to prove that there exists a constant $C$ such that
\begin{equation}
	\label{equ:goalredu0}
	\|S^{-\delta}(\mathcal{L}_{\mathbf{A}})\chi_E\|_{L^q(\mathbb{R}^2)}\leq C||\chi_E||_{L^p(\mathbb{R}^2)},
\end{equation}
for any characteristic function $\chi_E$ and $(1/p,1/q)\in (BB')$ and  $q>4$. Finally, by duality and interpolation in Lorentz spaces, the estimate \eqref{equ:goalredu0} implies \eqref{est:BR1} on $(BB')$.

By Proposition \ref{prop: ker-BR}, we obtain
\begin{align}\label{eq:log1}
	S^{-\delta}(\mathcal{L}_\mathbf{A})(x,y)=&\sum_{\pm}\big(G^{\pm}+D^{\pm}\big)(r_1,\theta_1;r_2,\theta_2)+R(r_1,\theta_1;r_2,\theta_2)
\end{align}
where the kernels are defined by
\begin{equation}
	\label{ker-G}
	\begin{split}
		G^{\pm}(r_1,\theta_1;r_2,\theta_2)&=e^{\pm i|x-y|} (1+|x-y|)^{-\frac32+\delta} a^1_\pm \times A_\alpha(\theta_1,\theta_2),\\
		D^\pm(r_1,\theta_1;r_2,\theta_2)&= \int_0^\infty e^{\pm i|{\bf n}_s|} (1+|{\bf n}_s|)^{-\frac32+\delta} a^1_\pm\times B_\alpha(s,\theta_1,\theta_2)\;ds,
	\end{split}
\end{equation}
and $R(r_1,\theta_1;r_2,\theta_2)$ is the term with lower order. Here $A_\alpha(\theta_1,\theta_2)$, $B_\alpha(s,\theta_1,\theta_2)$ are as in \eqref{A-al} and \eqref{B-al}. Hence, we write that
\begin{align*}
	S^{-\delta}(\mathcal{L}_\mathbf{A})f(x)
	=\big(T_{G^+}f+T_{G^-}f+T_{D^+}f+T_{D^-}f+T_{R}f\big)(x),
\end{align*}
where
$$T_{K}f(x)=\int_0^\infty \int_0^{2\pi }K(r_1,\theta_1;r_2,\theta_2) f(r_2,\theta_2)  d\theta_2 \,r_2 dr_2.$$
Thus, to prove \eqref{equ:goalredu0}, it suffices to prove that there exists a constant $C$ such that
\begin{equation}
	\label{equ:goalredu}
	\|T_{K}\chi_E\|_{L^q(\mathbb{R}^2)}\leq C||\chi_E||_{L^p(\mathbb{R}^2)}, \quad K\in\{G^\pm, D^\pm\}
\end{equation}
for any characteristic function $\chi_E$.

\section{The estimate of $ T_{G^\pm}$}
\label{sec:TG}
In this section, we prove
\begin{equation}
	\label{equ:tg1red}
	\|T_{G^\pm}\chi_E\|_{L^{q}(\mathbb{R}^2)}\leq C||\chi_E||_{L^{p}(\mathbb{R}^2)}
\end{equation}
if $(1/p,1/q)\in (BB') (q>4)$. Since $\pm$ is not essential in the proof, we only consider one case. To light the notation, we replace ${G^\pm}$ by $G$.

By dropping the factors $e^{\pm i\alpha\pi}$ and $e^{i\int_{\theta_1}^{\theta_2}A(\theta')d\theta'}$ in $A_\alpha(\theta_1,\theta_2)$, it suffices to prove
\begin{equation}\label{equ:tg2red}
	\begin{split}
		\Big\|\int_0^\infty \int_0^{2\pi } K^{\delta}(x-y)\mathbf{1}_{I}(|\theta_1-\theta_2|)\chi_E(r_2,\theta_2) r_2 dr_2 d\theta_2\Big\|_{L^q(\mathbb{R}^2)}\leq C\|\chi_E\|_{L^p(\mathbb{R}^2)}
	\end{split}
\end{equation}
where
\begin{equation*}
	\begin{split}
		K^{\delta}(x)&:=e^{i|x|} |x|^{-\frac32+\delta}
	\end{split}
\end{equation*}
and $I=[0,\pi]$ and $[\pi,2\pi)$. We will only consider the case that $I=[0,\pi]$. Indeed, if $I=[\pi, 2\pi)$, by replacing $\theta_1$ by $\theta_1\pm 2\pi$, the same argument works.

\begin{proposition}\label{prop:TGj}
Assume that $\psi$ is a smooth function with compact support away from $0$. For $\lambda \geq 1$, we define
\[ T^\lambda_G(f)(x) = \int_{\mathbb{R}^2} e^{2\pi i \lambda |x-y|} \psi(x-y)\mathbf{1}_{I}(|\theta_1-\theta_2|) f(y) dy.\]
Then, we have
\begin{equation}\label{equ:TGfest}
	\|T^\lambda_G f\|_{L^q(\mathbb{R}^2)} \leq C \lambda^{-2/q} \|f\|_{L^r(\mathbb{R}^2)} 
\end{equation}
if $q > 4$ and $\tfrac3q + \tfrac1r =1$, i.e. $q = 3r'$.
\end{proposition}

\begin{proof}[Proof of \eqref{equ:tg1red} assuming Proposition \ref{prop:TGj}]
Using the standard partition of unity of $B_1(0)^c$:
$$1=\sum_{j\geq1}\beta(2^{-j}r),\quad \beta\in C_c^\infty\big(\big[\tfrac38,\tfrac43\big]\big),$$
we decompose $K^{\delta}=\sum_{j}K^{\delta}_j$, where
$$K^{\delta}_j(x):=\beta(2^{-j}|x|)K^{\delta}(x), \quad j\geq1.$$
Let $T_{G}^j$ be the operator associated with kernel $K^{\delta}_j(x-y)\mathbf{1}_I(|\theta_1-\theta_2|)$ and let $q>4$.
Then by Proposition \ref{prop:TGj} and scaling, we know that for any $j\geqslant 1$
\begin{equation}\label{est:TGj}
\|T_{G}^j\|_{L^r\to L^q}\lesssim 2^{-j\left(\frac32-\delta\right)}2^{\frac{6j}q}
\end{equation}
if $\tfrac3q +\tfrac1r = 1$. And by duality, we know that 
$$\|T_{G}^j\|_{L^r\to L^q}\lesssim 2^{-j\left(\frac32-\delta\right)}2^{\frac{2j}q}$$ if $q>\frac{4}{3}$ and $\tfrac{1}{3q} + \tfrac1r = 1$.
	
Note that $q>4$. From the above two estimates applied to a characteristic function $f = \chi_E$, we obtain
\[ \| T^k_G \chi_E \|_q \leq C \min \left\{ 2^{-(3/2 - \delta)k} 2^{6k/q} |E|^{1 - 3/q}, \ 2^{-(3/2 - \delta)k} 2^{2k/q} |E|^{1 - 1/3q} \right\}. \]
Observe that the first term in the braces is smaller than the second precisely when $2^k < |E|^{2/3}$. Let $N$ be the integer such that $2^N < |E|^{2/3} \leq 2^{N+1}$. Then
\[ \sum_{k=1}^{\infty} \| T^k_G \chi_E \|_q \leq C \sum_{k=-\infty}^{N} 2^{k(-3/2 + \delta + 6/q)} |E|^{1 - 3/q} + C \sum_{k=N+1}^{\infty} 2^{k(-3/2 + \delta + 2/q)} |E|^{1 - 1/3q}. \]
Since $1/4 - \delta/6 < 1/q < 3/4 - \delta/2$, we have $-3/2 + \delta + 6/q > 0$, and $-3/2 + \delta + 2/q < 0$. So the two geometric series are convergent, and we obtain
\begin{align*}
\sum_{k=1}^{\infty} \|T^k_G \chi_E\|_q
&\leq C 2^{N(-3/2 + \delta + 6/q)} |E|^{1 - 3/q} + C 2^{(N+1)(-3/2 + \delta + 2/q)} |E|^{1 - 1/3q}\\
&\leq C |E|^{(2/3)(-3/2 + \delta + 6/q)} |E|^{1 - 3/q} + C |E|^{(2/3)(-3/2 + \delta + 2/q)} |E|^{1 - 1/3q}
\leq C |E|^{1/p}.
\end{align*}
Therefore, we conclude that 
$$\|T_G \chi_E\|_q \leq C \|\chi_E\|_p\quad \text{for}\quad (1/p,1/q)\in (BB'), q>4.$$
\end{proof}

Our main task is to prove Proposition \ref{prop:TGj} by making use of  the oscillatory integral theory by Stein \cite{Stein} and H\"ormander \cite{Hor}. We recall the following estimates by H\"ormander\cite[Theorem 1.4]{Hor} and also  Bak\cite[Appendix]{B}:

\begin{lemma}\label{lem:Bak's lemma}
Assume that $\psi$ is a smooth function with compact support away from $0$. For $\lambda \geq 1$, we define
\[ G_\lambda(f)(x) = \int_{\mathbb{R}^2} e^{2\pi i \lambda |x-y|} \psi(x-y) f(y) dy. \]
Then, there holds
\[ \|G_\lambda f\|_{L^q(\mathbb{R}^2)} \leq C \lambda^{-2/q} \|f\|_{L^r(\mathbb{R}^2)} \]
if $q > 4$ and $\tfrac3q + \tfrac1r = 1$.
\end{lemma}

Building upon Lemma \ref{lem:Bak's lemma}, Proposition \ref{prop:TGj} can be derived via a localized argument, as detailed in \cite{MYZ}. In the present work, however, we provide a significantly simpler proof, at the expense of an arbitrarily small loss in the power of $\lambda$—that is, a factor of $\lambda^\varepsilon$ for any $\varepsilon > 0$. It is straightforward to verify that such an $\varepsilon$-loss is acceptable within the proof of our main theorem.

The core of our approach lies in establishing a “stability lemma” for the underlying oscillatory integral operator. More precisely, compared with Lemma \ref{lem:Bak's lemma}, the kernel of $T_G^{\lambda}$ contains an additional jump function $\mathbf{1}_{[0,\pi]}(|\theta_1-\theta_2|)$. As a result, Proposition \ref{prop:TGj} may be viewed as a variant of the stability lemma adapted to this rough setting. To prove Proposition \ref{prop:TGj}, we note that
\[\mathbf{1}_{[0,2\pi]}(|\theta_1-\theta_2|)=\sum_{k\in\mathbb{Z}}\frac{\sin \left(\frac{k}{2}\pi\right)}{\pi k}e^{-ik(\theta_1-\theta_2)/2}.\]
Let $M$ be a free parameter to be chosen later, we have
\begin{align*}
	T_G^{\lambda}f(x)
	&=\sum_{|k|<M}\frac{\sin \left(\frac{k}{2}\pi\right)}{\pi k}e^{-ik\theta_1/2}\int_{\mathbb{R}^2} e^{2\pi i \lambda |x-y|} \psi(x-y)e^{ik\theta_2/2} f(y) dy+\int_{\mathbb{R}^2} R_M(x,y)f(y)dy,
\end{align*}
where
\[R_M(r_1,r_2;\theta_1,\theta_2):=e^{2\pi i \lambda |x-y|} \psi(x-y)\Big[\mathbf{1}_{I}(|\theta_1-\theta_2|)-\sum_{|k|<M}\frac{\sin \left(\frac{k}{2}\pi\right)}{\pi k}e^{-ik(\theta_1-\theta_2)/2}\Big].\]
Note that by Hausdorff-Young's inequality, we have
\[\bigg\|\mathbf{1}_{I}(|\theta|)-\sum_{|k|<M}\frac{\sin \left(\frac{k}{2}\pi\right)}{\pi k}e^{-ik\theta/2}\bigg\|_{L^p(-2\pi,2\pi)}\lesssim M^{-1/p}\]
for $2\leqslant p<+\infty$. By generalized Schur's test, we may take $M:=[\lambda^{C_{q,r}}]$ for some $C_{q,r}>0$ such that
\[\left\|\int_{\mathbb{R}^2} R_M(x,y)f(y)dy\right\|_{L^q(\mathbb{R}^2)}\lesssim \lambda^{-2/q}\|f\|_{L^r(\mathbb{R}^2)}.\]
And by Lemma \ref{lem:Bak's lemma} and the triangle inequality we have
\begin{align*}
	&\bigg\|\sum_{|k|<M}\frac{\sin \left(\frac{k}{2}\pi\right)}{\pi k}e^{-ik\theta_1/2}\int_{\mathbb{R}^2} e^{2\pi i \lambda |x-y|} \psi(x-y)e^{ik\theta_2/2} f(y) dy\bigg\|_{L^q(\mathbb{R}^2)}\\
	\lesssim& \log M\cdot \lambda^{-2/q}\|f\|_{L^r(\mathbb{R}^2)}\lesssim \lambda^{-2/q+\varepsilon}\|f\|_{L^r(\mathbb{R}^2)}.
\end{align*}
 Therefore, we finish the proof of  Proposition \ref{prop:TGj}.

\section{The estimation of $T_{D^{\pm}}$}\label{sec:TD}
This section is devoted to proving the bound
\begin{equation}\label{equ:td1red}
\|T_{D^\pm} \chi_E\|_{L^{q}(\mathbb{R}^2)} \leq C \|\chi_E\|_{L^{p}(\mathbb{R}^2)}
\end{equation}
for $(1/p, 1/q) \in (BB')$ and $q > 4$. The argument, though more complicated, is elementary.  We proceed by adapting the strategy developed in \cite{FZZ1}.

As before, we only consider the $+$ case. To light the notation, we replace ${D^\pm}$ by $D$ and only consider the $+$ case. Recall
\begin{equation}
\begin{split}
D(r_1,\theta_1;r_2,\theta_2)&= \frac1{\pi}\int_0^\infty e^{i|{\bf n}_s|} |{\bf n}_s|^{-\frac{3}{2}+\delta} a(|{\bf n}_s|)\, B_\alpha(s,\theta_1,\theta_2)\;ds
\end{split}
\end{equation}
for some smooth function $a$ satisfies 
\[\Big|\big(\tfrac{\partial}{\partial r}\big)^Na(r)\Big|\leqslant C_n r^{-N},\;r>0,\quad\forall N\geqslant 0.\]

It suffices to prove the three estimates, for $\ell=1,2,3$ and $(\frac{1}{p},\frac{1}{q})\in (BB')$ and $q>4$
\begin{equation}\label{equ:kdellred}
\Big\|\int_0^\infty \int_0^{2\pi } K^{\ell}_D(r_1,r_2;\theta_1-\theta_2) \chi_E(r_2,\theta_2) r_2 dr_2 d\theta_2\Big\|_{L^q(\mathbb{R}^2)}\leq C\|\chi_E\|_{L^p(\mathbb{R}^2)},
\end{equation}
where
\begin{align}\label{kerD12}
K^1_D(r_1,r_2;\theta_1-\theta_2)&=\int_0^\infty e^{ i|{\bf n}_s|} |{\bf n}_s|^{-\frac32+\delta}  a(|{\bf n}_s|)\, e^{-|\alpha|s}\;ds,\\\nonumber
K^2_D(r_1,r_2;\theta_1-\theta_2)&=\int_0^\infty e^{ i|{\bf n}_s|} |{\bf n}_s|^{-\frac32+\delta}  a(|{\bf n}_s|) \frac{(e^{-s}-\cos(\theta_1-\theta_2+\pi))\sinh(\alpha s)}{\cosh(s)-\cos(\theta_1-\theta_2+\pi)}\;ds,
\end{align}
and
\begin{equation}\label{kerD3}
\begin{split}
K^3_D(r_1,r_2;\theta_1-\theta_2)&=\int_0^\infty e^{ i|{\bf n}_s|} |{\bf n}_s|^{-\frac32+\delta}  a(|{\bf n}_s|)\frac{\sin(\theta_1-\theta_2+\pi)\cosh(\alpha s)}{\cosh(s)-\cos(\theta_1-\theta_2+\pi)}\;ds.
\end{split}
\end{equation}
Using the partition of unity again
$$\beta_0(r)=1-\sum_{j\geq1}\beta(2^{-j}r),\quad \beta\in\mathcal{C}_c^\infty\big(\big[\tfrac38,\tfrac43\big]\big),$$ we decompose
\begin{equation}\label{KDlj}
K^{\ell}_D(r_1,r_2;\theta_1-\theta_2)=:\sum_{j\geq 0}K^{\ell,j}_D(2^jr_1,2^jr_2;\theta_1-\theta_2),
\end{equation}
where for $j\geq1$ and $\ell=1,2,3$
$$K_D^{\ell,j}(r_1,r_2;\theta_1-\theta_2)=\beta(r_1+r_2)K^\ell_D(2^jr_1,2^jr_2;\theta_1-\theta_2)$$
and $K_D^{\ell,0}(r_1,r_2;\theta_1-\theta_2)=\beta_0(r_1+r_2)K^\ell_D(r_1,r_2;\theta_1-\theta_2)$. For our purpose, we need the following proposition, which is an analogue of \eqref{est:TGj}.

\begin{proposition}\label{prop:TDj}
For $\ell=1,2,3$, let $T_{D}^{\ell,j}$ be the operator associated with kernel $K_{D}^{\ell,j}(r_1,r_2;\theta_1-\theta_2)$ given in \eqref{KDlj}. Then, for $j\geq0$, it holds
\begin{equation}\label{est:TDj}
\begin{split}
\|T_{D}^{\ell,j}\|_{L^r\to L^q}&\leq C 2^{-j\left(3/2-\delta\right)}2^{-2j/q}
\end{split}
\end{equation}
for $q>4$ and $\tfrac3q +\tfrac1r=1$.
\end{proposition}
With  the Proposition \ref{prop:TDj} in hand, we obtain \eqref{equ:td1red} by using the same argument in the proof of \eqref{equ:tg1red} from Proposition \ref{prop:TGj},
see the details after Proposition \ref{prop:TGj}.

\vspace{0.2cm}

For $j=0$, from \eqref{equ:ream2}, \eqref{equ:ream3} and \eqref{equ:ream4}, we have
$$|K_D^{\ell,0}(r_1,r_2;\theta_1-\theta_2)|\leq \beta_0(r_1+r_2).$$
Then we conclude $ \|T_{D}^{\ell,0}\|_{L^r\to L^q}\leq C$ by generalized Schur’s test. And for $j\geq1$, we need the following proposition.
\begin{proposition}[\cite{FZZ1}]\label{lem:kerlplq}
Let $T_K$ be defined by
$$T_Kf(r_1,\theta_1):=\int_0^\infty \int_{0}^{2\pi} K(r_1,r_2,\theta_1,\theta_2)f(r_2,\theta_2)\;d\theta_2\;r_2\;dr_2,$$
and the kernel $K(r_1,r_2,\theta_1,\theta_2)$ satisfies
\begin{equation}\label{equ:kercond}
|K(r_1,r_2,\theta_1,\theta_2)|\lesssim \big(1+2^jr_1r_2\big)^{-\frac12}\beta(r_1+r_2).
\end{equation}
Then, there holds
\begin{equation}\label{equ:tkpqest}
\big\|T_Kf\big\|_{L^q(\mathbb{R}^2)}\lesssim 2^{-\frac{2j}{q}}\|f\|_{L^p(\mathbb{R}^2)}
\end{equation}
for $q>4$ and $q>p'$.
\end{proposition}

To prove \eqref{est:TDj}, it suffices to show the kernels $2^{j\left(\frac{3}{2}-\delta\right)}K_D^{\ell,j}(r_1,r_2;\theta_1-\theta_2)$ satisfy \eqref{equ:kercond}.

\begin{lemma}\label{lem:ker-est12}
 For  $j\geq1$ and $l=1,2,3$, then there holds
\begin{equation}\label{kerD-est}
\begin{split}
|K_{D}^{l,j}(r_1,r_2;\theta_1-\theta_2)|\lesssim 2^{-j\left(\frac{3}{2}-\delta\right)}\big(1+2^jr_1r_2\big)^{-\frac12}.
\end{split}
\end{equation}
\end{lemma}

\begin{proof}
To prove this, for fixed $r_1, r_2$, $\theta_1,\theta_2$ and $j\geq1$, we define
\begin{equation}
\begin{split}
\psi_1(r_1,r_2,\theta_1,\theta_2; s)&=|{\bf n}_s|^{-3/2+\delta} a(2^j|{\bf n}_s|)\, e^{-|\alpha|s},\\
\psi_2(r_1,r_2,\theta_1,\theta_2; s)&=|{\bf n}_s|^{-3/2+\delta} a(2^j|{\bf n}_s|)\, \frac{(e^{-s}-\cos(\theta_1-\theta_2+\pi))\sinh(\alpha s)}{\cosh(s)-\cos(\theta_1-\theta_2+\pi)},\\
\psi_3(r_1,r_2,\theta_1,\theta_2; s)&=|{\bf n}_s|^{-3/2+\delta} a(2^j|{\bf n}_s|)\, \frac{\sin(\theta_1-\theta_2+\pi)\cosh(\alpha s)}{\cosh(s)-\cos(\theta_1-\theta_2+\pi)}.
\end{split}
\end{equation}
Then by the definition, we have
\begin{equation}\label{KDl}
\begin{split}
K_D^{\ell,j}(r_1,r_2;\theta_1-\theta_2)&=2^{-j\left(\frac{3}{2}-\delta\right)}\beta(r_1+r_2) \int_0^\infty e^{i2^j|{\bf n}_s|} \psi_\ell(s)\;ds,\quad \ell=1,2,3.
\end{split}
\end{equation}
	
\textbf{The case $\ell=1$.} If $2^jr_1r_2\lesssim1$,  then \eqref{kerD-est} follows by using \eqref{equ:ream1}, \eqref{equ:ream2} and \eqref{equ:ream3}.
Thus, from now on, we assume $2^jr_1r_2\geq 1$ in the proof.
We first compute that
\begin{equation}
\begin{split}
\partial_s |{\bf n}_s|&=\frac{r_1r_2\sinh s}{(r_1^2+r_2^2+2r_1r_2\cosh s)^{1/2}}\\
\partial^2_s|{\bf n}_s|&=\frac{r_1r_2\cosh s}{(r_1^2+r_2^2+2r_1r_2\cosh s)^{1/2}}-\frac{(r_1r_2\sinh s)^2}{(r_1^2+r_2^2+2r_1r_2\cosh s)^{3/2}},
\end{split}
\end{equation}
therefore we obtain, for $0\leq s\leq 1$
\begin{equation}
\partial_s |{\bf n}_s|(0)=0,\quad \big|\partial^2_s |{\bf n}_s|\big|\geq c \frac{r_1r_2}{r_1+r_2} \gtrsim r_1r_2,
\end{equation}
and for $s\geq 1$
\begin{equation}
\partial_s|{\bf n}_s|\geq c \frac{r_1r_2}{r_1+r_2} \gtrsim r_1r_2.
\end{equation}
One can verify that $\partial_s|{\bf n}_s|$ is monotonic on the interval $[1,\infty)$ and the facts that
\begin{equation}\label{psi-b-1}
\int_0^\infty |\psi_1'(s)|ds\lesssim 1.
\end{equation}
Indeed, if the derivative hits on $|{\bf n}_s|^{-1/2} a(2^j|{\bf n}_s|)$ which is bounded, we again use \eqref{equ:ream1}, \eqref{equ:ream2} and \eqref{equ:ream3} to obtain \eqref{psi-b-1}. If the derivative hits on $e^{-|\alpha|s}$, it brings harmlessness. By using Van der Corput Lemma, thus we prove
\begin{equation}
\begin{split}
|K_{D}^{1,j}(r_1,r_2;\theta_1-\theta_2)|\lesssim 2^{-j\left(\frac{3}{2}-\delta\right)}\big(2^jr_1r_2\big)^{-\frac12}.
\end{split}
\end{equation}
	
\textbf{The case $\ell=2$.} We can bound $\int_1^{\infty} e^{i2^j|{\bf n}_s|} \psi_2(s)\;ds$ as above. However, we need more argument to bound $\int_0^1 e^{i2^j|{\bf n}_s|} \psi_2(s)\;ds$. To this aim, when $s$ is close to $0$, we replace
\begin{equation*}
\begin{split}
\frac{(e^{-s}-\cos(\theta_1-\theta_2+\pi))\sinh(\alpha s)}{\cosh(s)-\cos(\theta_1-\theta_2+\pi)}\sim \frac{(-s+b^2)(\alpha s)}{\frac{s^2}2+b^2}
\end{split}
\end{equation*}
where $b=\sqrt{2}\sin\big(\frac{\theta_1-\theta_2+\pi}2\big)$.
Then for $0\leq s\leq 1$,  uniformly in $b$, we have
\begin{equation}\label{d-v1}
\begin{split}
\Big|\partial_s^{k}\Big(\frac{(e^{-s}-\cos(\theta_1-\theta_2+\pi))\sinh(\alpha s)}{\cosh(s)-\cos(\theta_1-\theta_2+\pi)}-\frac{(-s+b^2)(\alpha s)}{\frac {s^2}2+b^2}\Big)\Big|\lesssim 1,\quad k=0,1
\end{split}
\end{equation}
which has been proved in the appendix of \cite{FZZ1}.
	
Therefore  the difference term is accepted by using Van der Corput Lemma again. Hence, we need to control
\begin{equation}\label{psi-2'1}
\begin{split}
\int_0^1 e^{i2^j|{\bf n}_s|}|{\bf n}_s|^{-3/2+\delta} a(2^j|{\bf n}_s|)\frac{(-s+b^2)(\alpha s)}{\frac{s^2}2+b^2}\;ds
\end{split}
\end{equation}
which is bounded by
\begin{equation}
\begin{split}
& \bigg(\Big|\int_0^1 e^{i2^j|{\bf n}_s|} |{\bf n}_s|^{-3/2+\delta} a(2^j|{\bf n}_s|)\;ds\Big|+\Big|\int_0^1 e^{i2^j|{\bf n}_s|} |{\bf n}_s|^{-3/2+\delta} a(2^j|{\bf n}_s|)\, \frac{b^2}{\frac{s^2}2+b^2}\;ds\Big|\\
&+\Big|\int_0^1 e^{i2^j|{\bf n}_s|} |{\bf n}_s|^{-3/2+\delta} a(2^j|{\bf n}_s|)\, \frac{sb^2}{\frac{s^2}2+b^2}\;ds\Big|\bigg).
\end{split}
\end{equation}
Now we can bound
\begin{equation}
\begin{split}
& \Big(\int_0^1 \Big|\partial_s\big(|{\bf n}_s|^{-3/2+\delta} a(2^j|{\bf n}_s|)\big)\Big|\;ds+\Big|\int_0^1 \partial_s\Big(|{\bf n}_s|^{-3/2+\delta} a(2^j|{\bf n}_s|)\, \frac{b^2}{\frac{s^2}2+b^2}\Big)\;ds\Big|\\
&+\Big|\int_0^1\partial_s\Big(|{\bf n}_s|^{-3/2+\delta} a(2^j|{\bf n}_s|)\, \frac{sb^2}{\frac{s^2}2+b^2}\Big)\;ds\Big|\Big)\lesssim 1.
\end{split}
\end{equation}
Then by using Van der Corput Lemma again, we obtain
\begin{equation}
\begin{split}
\Big|\int_0^1 e^{i2^j|{\bf n}_s|}|{\bf n}_s|^{-1/2} a(2^j|{\bf n}_s|)\frac{(-s+b^2)(\alpha s)}{\frac{s^2}2+b^2}\;ds\Big|\lesssim  (2^jr_1r_2)^{-\frac12}.
\end{split}
\end{equation}
	
\textbf{The case $\ell=3$.} Now we are left to consider $K_D^{3,j}$ which is more complicated.
To this end, we define
\begin{equation}\label{psi-m}
\begin{split}
\psi_{3,m}(r_1,r_2,\theta_1,\theta_2; s)&=(r_1+r_2)^{-3/2+\delta} a(2^j(r_1+r_2))\, \frac{\sin(\theta_1-\theta_2+\pi)}{\frac{s^2}2+b^2},
\end{split}
\end{equation}
and $\psi_{3,e}=\psi_{3}-\psi_{3,m}$.
We further define the kernels
\begin{equation}\label{KDe3}
\begin{split}
K_{D,m}^{3,j}(r_1,r_2;\theta_1-\theta_2)&=2^{-j\left(\frac{3}{2}-\delta\right)}\beta(r_1+r_2) \int_0^\infty e^{i2^j|{\bf n}_s|} \psi_{3, m}(s)\;ds,\\
K_{D,e}^{3,j}(r_1,r_2;\theta_1-\theta_2)&=2^{-j\left(\frac{3}{2}-\delta\right)}\beta(r_1+r_2) \int_0^\infty e^{i2^j|{\bf n}_s|} \psi_{3, e}(s)\;ds,
\end{split}
\end{equation}
and let
\begin{equation}\label{def-H}
H(r_1,r_2;\theta_1-\theta_2)=2^{-j\left(\frac{3}{2}-\delta\right)}\beta(r_1+r_2) (r_1+r_2)^{\frac{3}{2}-\delta} \int_0^\infty e^{i2^jr_1r_2 s^2} \psi_{3, m}(s)\;ds.
\end{equation}
We will show that
\begin{align}\label{kerDe3-est}
|K_{D,e}^{3,j}(r_1,r_2;\theta_1-\theta_2)|\lesssim& 2^{-j\left(\frac{3}{2}-\delta\right)}\big(1+2^jr_1r_2\big)^{-\frac12},\\\label{kerDm3-est}
\big|e^{-i2^j(r_1+r_2)}K_{D,m}^{3,j}(r_1,r_2;\theta_1-\theta_2)-H(r_1,r_2;\theta_1-\theta_2)\big|\lesssim& 2^{-j\left(\frac{3}{2}-\delta\right)}\big(1+2^jr_1r_2\big)^{-\frac12},\\\label{kerDH3-est}
|H(r_1,r_2;\theta_1-\theta_2)|\lesssim& 2^{-j\left(\frac{3}{2}-\delta\right)}\big(1+2^jr_1r_2\big)^{-\frac12}.
\end{align}
		
We first prove \eqref{kerDe3-est}. Arguing similarly as the case $\ell=1$, it suffices to show
\begin{equation}
\begin{split}
\int_0^1\big|\partial_s \psi_{3, e}(r_1,r_2,\theta_1,\theta_2; s)\big|ds\lesssim 1,
\end{split}
\end{equation}
uniformly in $r_1,r_2,\theta_1,\theta_2$, when $r_1+r_2\sim 1$. For our purpose, we have
\begin{align}\label{d-v2}
&\int_0^1\Big|\partial_s\Big[\Big(|{\bf n}_s|^{-3/2+\delta} a(2^j|{\bf n}_s|)\Big)\Big(\frac{\cosh(\alpha s)}{\cosh(s)-\cos(\theta_1-\theta_2+\pi)}-\frac{1}{\frac{s^2}2+b^2}\Big)\Big]\Big|\,ds\lesssim 1,
\end{align}
and
\begin{align}\label{d-v3}
\int_0^1\Big|\partial_s\Big[\Big(|{\bf n}_s|^{-3/2+\delta} a(2^j|{\bf n}_s|)-(r_1+r_2)^{-\frac12}a(2^j(r_1+r_2))\Big)\frac{1}{\frac{s^2}2+b^2}\Big]\Big|\, ds\lesssim& 1,
\end{align}
which are verified in appendix of \cite{MYZ}.
		
We next prove \eqref{kerDm3-est}. We will use the the Morse Lemma to write the phase function in term of quadratic formula via making variable change. Let 
$$\bar{\varphi}(s)=\frac{\big(r_1^2+r_2^2+2r_1r_2\cosh s\big)^{\frac12}}{r_1r_2}-\frac{r_1+r_2}{r_1r_2},$$
then $\bar{\varphi}(0)=\bar{\varphi}'(0)=0$ and $\bar{\varphi}''(0)=\frac{1}{r_1+r_2}\neq 0$.
Let 
$$g(s)=2\int_0^1(1-t)\bar{\varphi}''(ts)dt,$$
then we can write $\bar{\varphi}(s)=\frac12g(s)s^2$. Note that $g(0)=\bar{\varphi}''(0)\neq 0$, we make the variable changing $\tilde{s}=|g(s)|^{\frac12}s$.
Hence
\begin{equation}\label{d-ss}
\frac{d\tilde{s}}{ds}=(r_1+r_2)^{-\frac12}+O(r_1r_2s^2),\quad \frac{d{s}}{d\tilde{s}}=(r_1+r_2)^{\frac12}+O(r_1r_2\tilde{s}^2),
\end{equation}
and
$$\partial_{\tilde{s}}s\big|_{s=0}=|\bar{\varphi}''(0)|^{-\frac12}.$$
We write
$$\varphi(s)-(r_1+r_2)=r_1r_2\bar{\varphi}(s)=r_1r_2\tilde{s}^2,$$
by \eqref{d-ss}, hence
\begin{align}
&\int_0^\infty e^{i2^j(|{\bf n}_s|-(r_1+r_2))}\psi_{3, m}(s)\;ds\\\nonumber
=& \int_0^\infty e^{i2^jr_1r_2\tilde{s}^2}\big(\psi_{3, m}(\tilde{s})+O(r_1r_2\tilde{s}^2\big)\big((r_1+r_2)^{\frac12}+O(r_1r_2\tilde{s}^2)\big) d\tilde{s}.
\end{align}
By using Van der Corput Lemma as before, as desired, the difference is bounded by
$$\big|e^{-i2^j(r_1+r_2)}K_{D,m}^{3,j}(r_1,r_2;\theta_1-\theta_2)-H(r_1,r_2;\theta_1-\theta_2)\big|\lesssim 2^{-j\left(\frac{3}{2}-\delta\right)}\big(1+2^jr_1r_2\big)^{-\frac12}.$$
		
We finally prove \eqref{kerDH3-est}. Recall the definitions \eqref{def-H} and \eqref{psi-m}, by scaling, it suffices to show
\begin{equation}
\Big|\int_0^\infty e^{i2^{j+1}r_1r_2 s^2} \, \frac{b}{s^2+b^2}\;ds\Big|\lesssim \big(1+2^jr_1r_2\big)^{-\frac12}.
\end{equation}
This is the same to term $H$ in \cite[(35)]{BFM18} so that it follows from the same argument.
		
\end{proof}


\end{document}